\documentclass[reqno]{amsart}
\usepackage{amsmath,amssymb,latexsym, url,newtxtext,newtxmath}
\usepackage[margin=1in]{geometry}
\usepackage{enumitem}
\usepackage[colorlinks=true,linkcolor=blue,citecolor=blue,urlcolor=blue]{hyperref}

\newtheorem{theorem}{Theorem}[section]
\newtheorem{lemma}[theorem]{Lemma}
\newtheorem{proposition}[theorem]{Proposition}

\theoremstyle{definition}
\newtheorem{definition}[theorem]{Definition}
\newtheorem{remark}[theorem]{Remark}

\DeclareMathOperator{\SO}{SO} 
\DeclareMathOperator{\tr}{tr}
\DeclareMathOperator{\dist}{dist}
\DeclareMathOperator{\diag}{diag}
\DeclareMathOperator{\dbsup}{db\text{-}sup}
\newcommand{\R}{\mathbb{R}}
\newcommand{\Z}{\mathbb{Z}}
\newcommand{\Q}{\mathbb{Q}}
\newcommand{\Sph}{\mathbb{S}} 
\newcommand{\Dom}{\mathcal{D}} 

\begin{document}

\title{A computer-assisted proof of Kuperberg's six-cylinder conjecture}
\author[I. Mati\'c]{
Ivan Mati\'c}
\address{
Baruch College, City University of New York}\email{ivan.matic@baruch.cuny.edu}
\author[R. Radoi\v ci\'c]{
Rado\v s Radoi\v ci\'c} \address{
Baruch College, City University of New York}\email{rados.radoicic@baruch.cuny.edu}

\begin{abstract}
We prove that at most six pairwise non-overlapping infinite unit cylinders can simultaneously touch
a unit ball. Six is achievable, so the answer to Kuperberg's question is exactly six. The proof is
computer-assisted in the following sense: the statement is reduced to a finite case analysis with
$2\,954\,984$ cases, and in each case a program verifies that one of three inequalities between
explicit polynomials with rational coefficients holds throughout a box. Each case is an elementary
arithmetic check, and the complete list of cases is provided with the paper.
\end{abstract}

\maketitle

\section{Introduction}\label{sec:intro}
\subsection{Formulation}\label{subsection:formulation}

Throughout, a \emph{cylinder} means a closed solid right circular cylinder of infinite length and
radius $1$. Two cylinders are \emph{non-overlapping} if their interiors are disjoint, and a
cylinder \emph{touches} the unit ball $B$ if it meets $B$ but not the interior of $B$. In 1990
W.~Kuperberg asked how many pairwise non-overlapping cylinders can touch a fixed unit ball, and
conjectured that the answer is six \cite{K}. We prove that it is.
Our proof has a computer-assisted component that does case analysis. The source code and the
printout of cases can be found at \cite{MRCode2026}.

\begin{theorem} \label{thm:main}
At most six pairwise non-overlapping infinite unit cylinders can simultaneously touch a unit ball.
\end{theorem}

The bound of Theorem~\ref{thm:main} is attained and the
answer to the question is exactly six, since arrangements of six cylinders are easy to construct.  
One such configuration is constructed as follows. Let 
$B$ be the unit ball centered at the origin, and let 
$A_1A_2\dots A_6$ be a regular hexagon of edge length 
$2$ in the 
$Oxy$ plane, centered at $O$. For 
$i=1$, $\dots$, $6$, let $Z_i$ be the unit cylinder whose axis passes through $A_i$ and is parallel to the 
$Oz$ axis. Since each $A_i$ lies at distance 
$2$ from $O$ and consecutive vertices are at distance 
$2$ from each other, each $Z_i$ touches $B$ as well as its two neighboring cylinders, while the interiors of all seven bodies are mutually disjoint.

\subsection{On the use of AI}\label{subsection:AI}

Some of the ideas in this paper originated in work with Anthropic's Claude. The first proof we
obtained that way appeared to be correct, but it drew at once on representation theory, harmonic analysis,
spherical geometry and combinatorial optimization. Few readers hold that set of tools
simultaneously. We reconstructed the argument to make a much simpler proof. What
remains is an explicit table of integers, six finite checks on it, and one branch-and-bound
computation. None of it need be taken on trust: neither the theory the data came from, nor our
implementation of the computation.

The price is paid in Section \ref{sec:cert}, where matrices and weights appear with no motivation and no
derivation. They are simply data, and every property the proof uses is established directly from
them. The reasoning that produced them is sketched in Section \ref{sec:motivation}, which is
deliberately brief.

\subsection{History of the problem and overview of the literature} \label{subsection:history}

Kuperberg posed the question at a DIMACS workshop in 1990, together with several arrangements of
six unit cylinders touching a ball \cite{K}. The first upper bound is due to Heppes and
Szab\'o \cite{HS}, who showed that at most eight cylinders can touch. Bra\ss{} and Wenk \cite{BW}
improved this to at most seven, by a short argument comparing areas on a concentric sphere.
Whether seven is possible has been open since; Theorem~\ref{thm:main} answers it.

Much of the work since has concerned the radius rather than the number. The familiar configuration
of six unit cylinders is not rigid, and it was for a while natural to guess that six cylinders of
radius greater than $1$ could not touch the ball at all. Firsching \cite{F} showed otherwise,
finding by numerical exploration of the $18$-dimensional configuration manifold six cylinders of
radius $1.049659$ that do. Ogievetsky and Shlosman \cite{OS} later exhibited a
$\mathbb{D}_3$-symmetric configuration with
\[
r=\tfrac18\bigl(3+\sqrt{33}\bigr)\approx1.093070,
\]
and conjectured that this radius is the largest possible. In the opposite direction, Yardimci and
Bezdek \cite{Yar} proved that seven cylinders of radius $1.04965$, which is Firsching's value,
cannot touch the ball; at that one radius the analogue of Kuperberg's question is therefore
settled. The case of radius exactly $1$, which is the case Kuperberg asked about, is the one
treated here.

The problem is a cousin of the kissing number question, where twelve unit balls can touch a unit
ball and thirteen cannot \cite{SW}. Bounds of that kind are the natural home of Delsarte's linear
programming method and of its semidefinite refinement \cite{BV}, and it is a bound of that type,
transported to the rotation group, that we use; Section \ref{sec:motivation} says more about where it
comes from.

\subsection{A discontinuity}\label{subsection:discontinuity}

One feature of the problem shaped the proof below, and it is worth isolating before the argument
begins.

Let us consider the hexagon of Section~\ref{subsection:formulation} for which $A_1=(2,0,0)$ and
$A_2=\bigl(1,\sqrt3,0\bigr)$. We will fix the cylinder $Z_1$ and analyze the rotations of 
$Z_2$ by a small angle $\varepsilon$ about each of $Ox$ and $Oy$. Rotations about $O$ carry lines 
tangent to the sphere of radius $2$ to lines
tangent to it, so the rotated cylinder still touches $B$. The two axes of the cylinders $Z_1$ and $Z_2$ 
are initially at distance exactly 
$2$. After a rotation about $Ox$ they are at distance $1$; after a rotation about $Oy$, at distance
$\sqrt3$. Both hold for every $\varepsilon\ne0$. The nearest points of the two axes lie at distance of order $1/\varepsilon$ from the origin. As the direction of the tilt varies,
the limiting distance fills the whole interval $[0,2]$.  

The distance between two axes is therefore not a continuous function of the configuration. It is
discontinuous exactly where the axes are parallel, which is where the extremal configurations live.
This is a set of measure zero. 

The contrast with the kissing number problem, mentioned above, is sharp. There the condition that
two of the balls do not overlap is a polynomial inequality in their centers, and Delsarte's method
applies to it directly.

In our problem there is no such starting point. The condition on a pair of cylinders is that the
distance between their axes be at least $2$, and that distance is the discontinuous function
described above. Before a bound of Delsarte type can even be set up, the set of admissible pairs
has to be given an algebraic description of its own. Theorem~\ref{thm:domainQ} provides one: a
pair is admissible precisely when two explicit polynomials with integer coefficients are both
non-negative.

\section{The configuration space}\label{sec:reduction}

Fix the unit ball $B$ centered at $O$. Throughout, $u_0=(0,0,1)$, $v_0=(1,0,0)$, $a_1=2u_0=(0,0,2)$.

A solid unit cylinder touches the ball $B$ with center $O$ if and only if the axis
of the cylinder is tangent to the sphere $\Sph_2$ of radius $2$. Two solid unit cylinders have
disjoint interiors if and only if the distance between their axes is $\ge2$.

Recall that $\SO(3)$ is the group of all rotations of the three-dimensional Euclidean space around
the origin. Every cylinder that touches $B$ is uniquely determined by its axis
$\ell=\{2u+tv:t\in\R\}$ with $u,v\in\Sph^2$, $u\perp v$; here $2u$ is the point of tangency and
$v$ a direction of the axis, unique up to sign. There exists a unique $g\in\SO(3)$ such that
$u=gu_0,\ v=gv_0$. We will write $\ell=g\cdot\ell_0$ with $\ell_0=\{2u_0+tv_0\}$.

\begin{definition}\label{def:D}
We will define the \emph{safe zone} $\Dom$ as a subset of $\SO(3)$ in the following way
\[ \Dom=\{g\in\SO(3):\dist(\ell_0,\,g\cdot\ell_0)\ge2\}.\]
\end{definition}

\begin{definition}\label{def:nonOverlappingX}
A set $X\subseteq \SO(3)$ will be called \emph{non-overlapping} if
$g^{-1}g'\in\Dom$ whenever $g$ and $g'$ are two distinct elements of $X$.
\end{definition}

Consider $N$ pairwise non-overlapping cylinders that touch $B$, with axes
$\ell_1,\dots,\ell_N$, and choose $g_i\in\SO(3)$ with $\ell_i=g_i\cdot\ell_0$. The $g_i$ are
distinct, since $g$ determines $\ell=g\cdot\ell_0$. Rotations are isometries, so
$\dist(\ell_i,\ell_j)=\dist(\ell_0,\,g_i^{-1}g_j\cdot\ell_0)$, and the cylinders having disjoint
interiors is exactly the condition $g_i^{-1}g_j\in\Dom$ for all $i\ne j$. Thus
$\{g_1,\dots,g_N\}$ is a non-overlapping subset of $\SO(3)$ with $N$ elements.
Let $N_{\max}$ be the maximal number of elements that a non-overlapping subset of $\SO(3)$ can
have. Theorem \ref{thm:main} will be proved if we can establish $N_{\max}<7$.

\section{Delsarte--Bochner supremum}\label{sec:DBSequence}

Throughout this section, $L\in\mathbb N$ and, for $\ell=1,\dots,L$, we fix
\begin{enumerate}[leftmargin=2.2em,label=(\roman*)]
\item a symmetric positive-definite matrix $D_\ell\in\R^{n_\ell\times n_\ell}$, and
\item a map $T^\ell:\SO(3)\to\R^{n_\ell\times n_\ell}$ satisfying, for all $g_1,g_2,g\in\SO(3)$,
\[
T^\ell(g_1g_2)=T^\ell(g_1)\,T^\ell(g_2)
\qquad\text{and}\qquad
T^\ell(g)^\top D_\ell\, T^\ell(g)=D_\ell .
\]
\end{enumerate}

\begin{lemma}\label{lem:Torth}
Under (i)--(ii), for every $g\in\SO(3)$: $\det T^\ell(g)=\pm1$; $T^\ell(e)=I$; and
\[
T^\ell(g^{-1})\;=\;T^\ell(g)^{-1}\;=\;D_\ell^{-1}\,T^\ell(g)^\top D_\ell .
\]
\end{lemma}
\begin{proof}
Taking determinants in $T^\top D T=D$ gives $(\det T)^2\det D=\det D$, and $\det D>0$, so
$\det T=\pm1$; in particular every $T^\ell(g)$ is invertible. Then $T(e)=T(e\cdot e)=T(e)^2$ with
$T(e)$ invertible forces $T(e)=I$, whence $T(g)T(g^{-1})=T(e)=I$ gives $T(g^{-1})=T(g)^{-1}$.
Finally, multiplying $T^\top DT=D$ on the left by $D^{-1}$ and on the right by $T^{-1}$ gives
$D^{-1}T^\top D=T^{-1}$.
\end{proof}

\begin{definition}\label{def:DB}
For a sequence $\overrightarrow C=(C_\ell)_{\ell=1}^{L}$ of symmetric positive-semidefinite matrices
$C_\ell\in\R^{n_\ell\times n_\ell}$, we define its Delsarte--Bochner supremum as
\[
\dbsup\bigl(\overrightarrow C\bigr)\;=\;\sup_{g\in\Dom}\;\sum_{\ell=1}^{L}
\tr\!\Bigl(D_\ell^{-1}C_\ell\,T^\ell(g)\Bigr).
\]
\end{definition}

\begin{remark}\label{rem:entrywise}
When $D_\ell=\diag(w^\ell_1,\dots,w^\ell_{n_\ell})$ is
diagonal and $C_\ell$ is symmetric, the correct pairing has the explicit entrywise form
\[
\tr\!\Bigl(D_\ell^{-1}C_\ell\,T^\ell(g)\Bigr)
=\sum_{i,j}\frac{(C_\ell)_{ij}}{w^\ell_j}\;T^\ell(g)_{ij},
\]
the sum of the entries of the Schur product of $C_\ell$ and $T^\ell(g)$, with the $j$-th column
weighted by $1/w^\ell_j$. In particular
$\tr(D_\ell^{-1}C_\ell\,T^\ell(e))=\tr(D_\ell^{-1}C_\ell)=\sum_i (C_\ell)_{ii}/w^\ell_i$.
\end{remark}

\begin{theorem}\label{thm:dbBound}
Let $\overrightarrow C=(C_\ell)_{\ell=1}^{L}$ be a sequence of symmetric positive-semidefinite
matrices $C_\ell\in\R^{n_\ell\times n_\ell}$. If $\dbsup(\overrightarrow C)<0$ and $M$ is any real
number with $\dbsup(\overrightarrow C)\le M<0$, then
\begin{equation}\label{eq:dbBound}
N_{\max}\;\le\;1+\left\lfloor
\frac{\sum_{\ell=1}^{L}\tr\bigl(D_\ell^{-1}C_\ell\bigr)}{-M}\right\rfloor .
\end{equation}
\end{theorem}
\begin{proof}
Let $g_1,\dots,g_N$ be distinct elements of a non-overlapping set
(Definition~\ref{def:nonOverlappingX}), and put $S_\ell=\sum_{i=1}^{N}T^\ell(g_i)$. By
Lemma~\ref{lem:Torth},
\[
T^\ell\bigl(g_i^{-1}g_j\bigr)=T^\ell(g_i)^{-1}T^\ell(g_j)
=D_\ell^{-1}\,T^\ell(g_i)^\top D_\ell\,T^\ell(g_j),
\]
so that, summing over all ordered pairs and using linearity and cyclicity of the trace,
\begin{align*}
\sum_{i=1}^{N}\sum_{j=1}^{N}\tr\!\Bigl(D_\ell^{-1}C_\ell\,T^\ell\bigl(g_i^{-1}g_j\bigr)\Bigr)
&=\sum_{i,j}\tr\!\Bigl(D_\ell^{-1}C_\ell D_\ell^{-1}\,T^\ell(g_i)^\top D_\ell\,T^\ell(g_j)\Bigr)\\
&=\tr\!\Bigl(\bigl(D_\ell^{-1}C_\ell D_\ell^{-1}\bigr)\bigl(S_\ell^\top D_\ell S_\ell\bigr)\Bigr)
\;\ge\;0,
\end{align*}
since the trace of a product of two symmetric positive-semidefinite matrices is non-negative
(for such $P,Q$: $\tr(PQ)=\tr(\sqrt Q\,P\sqrt Q\,)\ge0$, the trace of a positive-semidefinite
matrix): $D_\ell^{-1}C_\ell D_\ell^{-1}$ is a congruence of $C_\ell$, and
$S_\ell^\top D_\ell S_\ell$ a congruence of $D_\ell$, so both are symmetric positive semidefinite.

Summing over $\ell$ and separating the $N$ diagonal pairs $i=j$ (where
$T^\ell(e)=I$ gives the value $\sum_\ell\tr(D_\ell^{-1}C_\ell)$) from the $N(N-1)$ off-diagonal
pairs (where $g_i^{-1}g_j\in\Dom$ by Definition~\ref{def:nonOverlappingX}, so each term is at most
$\dbsup(\overrightarrow C)\le M$),
\[
0\;\le\;N\sum_{\ell=1}^{L}\tr\bigl(D_\ell^{-1}C_\ell\bigr)\;+\;N(N-1)\,M .
\]
Since $M<0$, dividing by $N$ and rearranging gives
$(N-1)(-M)\le\sum_\ell\tr(D_\ell^{-1}C_\ell)$, which implies \eqref{eq:dbBound}.
\end{proof}

\section{Strategy for the proof}\label{sec:Strategy}

We will use the word \emph{certificate} for an explicit choice of matrices $D_{\ell}$, $C_{\ell}$,
and the maps $T^{\ell}$ for which we will be able to verify the conditions of Section \ref{sec:DBSequence}
and of Theorem~\ref{thm:dbBound}. Once those conditions are verified, the theorem gives the bound
\eqref{eq:dbBound}. In the case of our certificate, the bound becomes $N_{\max}\le6$.

We will apply Theorem~\ref{thm:dbBound} with $L=4$ and $M=-\tfrac{19}{20}$.

In Section \ref{sec:cert} we construct, for $\ell=1,\dots,4$, the following:
\begin{itemize}\item Explicit \emph{rational} diagonal matrices
$D_\ell=\diag(w^\ell_1,\dots,w^\ell_{2\ell+1})$ with positive diagonal entries,
\item Explicit symmetric
\emph{rational} matrices $C_\ell$, and 
\item Explicit matrices $T^\ell_{\mathrm h}$ of polynomials in $q_0$, $q_1$, $q_2$, $q_3$. 
\end{itemize} 
The arguments of $T^{\ell}$, which are the elements of $\SO(3)$, will be represented with
quaternions. The precise conditions the certificate must satisfy are the hypotheses
\ref{H:hom}--\ref{H:cover} of Theorem~\ref{thm:polyform} in Section~\ref{sec:quat}. In Section
\ref{sec:cert}, we verify that the provided matrices satisfy \ref{H:hom}--\ref{H:trace}. In 
Section~\ref{sec:cover} we construct a computer-assisted proof for \ref{H:cover}.

\section{Reduction to a polynomial problem}\label{sec:quat}
We will now find an equivalent problem in which $\SO(3)$ does not appear, and whose only objects
are polynomials in the four variables $q_0$, $q_1$, $q_2$, and $q_3$.

\subsection{Quaternions and rotations}\label{subs:quaternionsBasic}
For $q=(q_0,q_1,q_2,q_3)\in\R^4$ define
\[
R(q)=\left[\begin{array}{ccc}
q_0^2+q_1^2-q_2^2-q_3^2 & 2(q_1q_2-q_0q_3) & 2(q_1q_3+q_0q_2)\\
2(q_1q_2+q_0q_3) & q_0^2-q_1^2+q_2^2-q_3^2 & 2(q_2q_3-q_0q_1)\\
2(q_1q_3-q_0q_2) & 2(q_2q_3+q_0q_1) & q_0^2-q_1^2-q_2^2+q_3^2
\end{array}\right],
\]
whose entries are integer quadratic forms in $q$, and write
$|q|^2=q_0^2+q_1^2+q_2^2+q_3^2$. Let $p*q$ denote the Hamilton product,
\begin{align*}
(p*q)_0&= p_0q_0-p_1q_1-p_2q_2-p_3q_3,\\
(p*q)_1&=p_0q_1+q_0p_1+p_2q_3-p_3q_2,\\
(p*q)_2&=p_0q_2+q_0p_2+p_3q_1-p_1q_3,\\
(p*q)_3&=p_0q_3+q_0p_3+p_1q_2-p_2q_1.
\end{align*}

\begin{lemma}\label{lem:quat}
The following hold as polynomial identities in $\Z[p,q]$:
\begin{enumerate}[leftmargin=2.2em,label=(\alph*)]
\item $R(q)^\top R(q)=|q|^4 I$ and $\det R(q)=|q|^6$; in particular $R(q)\in\SO(3)$ whenever
$|q|=1$;
\item $R(p*q)=R(p)\,R(q)$ and $|p*q|^2=|p|^2|q|^2$;
\item the ten products $q_aq_b$ $(0\le a\le b\le3)$ are $\Z$-linear combinations of $|q|^2$ and the
entries $r_{ij}$ of $R(q)$; explicitly
$4q_0^2=|q|^2+r_{11}+r_{22}+r_{33}$, $4q_1^2=|q|^2+r_{11}-r_{22}-r_{33}$,
$4q_2^2=|q|^2-r_{11}+r_{22}-r_{33}$, $4q_3^2=|q|^2-r_{11}-r_{22}+r_{33}$,
$4q_0q_1=r_{32}-r_{23}$, $4q_0q_2=r_{13}-r_{31}$, $4q_0q_3=r_{21}-r_{12}$,
$4q_1q_2=r_{12}+r_{21}$, $4q_1q_3=r_{13}+r_{31}$, $4q_2q_3=r_{23}+r_{32}$.
\end{enumerate}
\end{lemma}
\begin{proof}
Each of the identities above is an equality between polynomials with integer coefficients. The degrees of the polynomials 
are at most $6$ in the
four variables $q$ in (a) and (c). The degrees are at most $4$ in the eight variables $p,q$ in (b).
Each identity can be directly verified by expanding both sides and comparing coefficients.
\end{proof}

\begin{lemma}\label{lem:cover2}
\leavevmode
\begin{enumerate}[leftmargin=2.2em,label=(\alph*)]
\item If $q,q'$ are unit quaternions with $R(q)=R(q')$, then $q'=\pm q$.
\item Every $g\in\SO(3)$ equals $R(q)$ for some unit quaternion $q$.
\end{enumerate}
\end{lemma}
\begin{proof} (a) By Lemma~\ref{lem:quat}(c) the entries of $R(q)$ together with $|q|^2=1$ determine every product
$q_aq_b$, i.e.\ the symmetric rank-one matrix $qq^\top$. If $qq^\top=q'q'^\top$ then, taking any
index $a$ with $q_a\ne0$ (such an index exists because $|q|=1$), the $a$-th columns give
$q_a\,q=q'_a\,q'$ with $q_a^2=(q'_a)^2$, so $q'=\pm q$.

\noindent (b) First, every $g\in\SO(3)$ has an eigenvector corresponding to eigenvalue $1$. Indeed, since
$\det g=1$ and $g^\top g=I$,
\[
\det(g-I)=\det(g)\det(I-g^\top)=\det\bigl((I-g)^\top\bigr)=\det(I-g)=(-1)^3\det(g-I).
\]
Therefore, $\det(g-I)=0$ and there is a unit vector $n$ with $gn=n$. Since $g$ is orthogonal and
fixes $n$, it maps the plane $n^\perp$ to itself, and its restriction there is an
orientation-preserving linear isometry of a plane, hence a rotation by some angle $\theta$. Let
$n=(n_x,n_y,n_z)$ and set
\[q=\left(\cos\tfrac\theta2,\;n_x\sin\frac\theta2,\;n_y\sin\frac\theta2,\;n_z\sin\frac\theta2\right).\]
Clearly, $q$ is a unit quaternion. A direct computation from the definition of $R(q)$ shows
$R(q)n=n$ and that $R(q)$ rotates $n^\perp$ by $\theta$ with the same orientation. Therefore,
$R(q)=g$.
\end{proof}

\subsection{The safe zone as a pair of polynomial inequalities}\label{subsec:safezoneQ}

By Lemma~\ref{lem:cover2} the map $q\mapsto R(q)$ carries the unit sphere of $\R^4$ onto
$\SO(3)$, two-to-one. To transport the safe zone $\Dom$ of Definition~\ref{def:D} along this map we
must first express the condition $\dist(\ell_0,g\cdot\ell_0)\ge2$ algebraically in $g$. This is
the purpose of the present subsection. The quaternions $q$ will be introduced in the next subsection.

Fix $g\in\SO(3)$ and write
\[
d_1=v_0,\qquad d_2=gv_0,\qquad a_2=2gu_0,\qquad w=a_2-a_1,
\qquad n=d_1\times d_2,\qquad N_2=|n|^2 .
\]
Thus $\ell_0$ passes through $a_1$ with unit direction $d_1$, and $g\cdot\ell_0$ passes through
$a_2$ with unit direction $d_2$. Writing $c=d_1\cdot d_2$, we have $N_2=1-c^2$, so $N_2>0$ exactly
when $d_1,d_2$ are linearly independent, that is, when the two lines are skew or intersecting, and
$N_2=0$ exactly when they are parallel.

\begin{lemma}\label{lem:projQ}
Suppose $N_2>0$, and let $\overline{w}$ be the orthogonal projection of $w$ onto
$\mathrm{span}\{d_1,d_2\}$. Then $N_2\,|\overline{w}|^{2}=\Gamma(w)$, where $\Gamma(w)$ is defined by 
\begin{align}
\Gamma(w)&=(w\cdot d_1)^{2}-2(w\cdot d_1)(w\cdot d_2)(d_1\cdot d_2)+(w\cdot d_2)^{2}. \label{eq:GammaQ}
\end{align}
\end{lemma}
\begin{proof}
Write $\overline{w}=s\,d_1+t\,d_2$ and put $b_1=w\cdot d_1$, $b_2=w\cdot d_2$. Since
$w-\overline{w}$ is orthogonal to $d_1$ and to $d_2$, taking inner products with $d_1,d_2$ gives
\[
\left[
\begin{array}{cc}1&c\\ c&1\end{array}\right]\left[\begin{array}{c}s\\ t\end{array}\right]
=\left[\begin{array}{c}b_1\\ b_2\end{array}\right],
\]
whose determinant is $1-c^{2}=N_2>0$; Cramer's rule gives $N_2\,s=b_1-cb_2$ and
$N_2\,t=b_2-cb_1$. Also $w-\overline{w}\perp \overline{w}$, so
$|\overline{w}|^{2}=w\cdot \overline{w}=s\,b_1+t\,b_2$. Multiplying by $N_2$ and substituting,
\[
N_2|\overline{w}|^{2}=(b_1-cb_2)b_1+(b_2-cb_1)b_2=b_1^{2}-2c\,b_1b_2+b_2^{2}=\Gamma(w). \qedhere
\]
\end{proof}

The distance from the point $a_2$ to the line $\ell_0$ is $U_1$ given by 
\begin{align}
U_1^2&=|w|^{2}-(w\cdot d_1)^{2}. \label{eq:UQ}
\end{align}

\begin{lemma}\label{lem:distQ}
With the notation above,
\[
\dist^{2}(\ell_0,\,g\cdot\ell_0)=
\begin{cases}
|w|^{2}-\Gamma(w)/N_2, & N_2>0,\\[2pt]
U_1^{2}, & N_2=0 .
\end{cases}
\]
Moreover $\Gamma(w)=0$ whenever $N_2=0$.
\end{lemma}
\begin{proof}
\emph{Case $N_2>0$.} The distance between $\ell_0$ and $g\cdot \ell_0$ is 
precisely the distance from the point $a_2$ to the plane through $\ell_0$ parallel to $g\cdot \ell_0$. This is the distance from $w$ to $\mathrm{span}\{d_1,d_2\}$. 
 Lemma~\ref{lem:projQ} and the Pythagorean theorem directly imply the desired result.

\emph{Case $N_2=0$.} Here $d_2=\pm d_1$, so $w\cdot d_2=\pm(w\cdot d_1)$ and
$d_1\cdot d_2=\pm1$ with matching signs; substituting into \eqref{eq:GammaQ},
\[
\Gamma(w)=(w\cdot d_1)^{2}-2(w\cdot d_1)^{2}+(w\cdot d_1)^{2}=0 .
\]
Moreover $w\perp d_1$. Indeed $a_1\cdot d_1=2u_0\cdot v_0=0$, and since $g$ is
orthogonal and $d_2=\pm d_1$ we have
$a_2\cdot d_1=2(gu_0)\cdot(\pm gv_0)=\pm2\,(u_0\cdot v_0)=0$. The two lines are
parallel, so the distance between them is the distance from $a_2$ to $\ell_0$,
which is $|w|$. Since $w\cdot d_1=0$, the required square of the distance is $|w|^{2}=|w|^{2}-(w\cdot d_1)^{2}=U_1^{2}$.
\end{proof}

Let us define
\begin{align}
Q&=\bigl(|w|^{2}-4\bigr)N_2-\Gamma(w). \label{eq:QQ}
\end{align} 

Observe that  
\begin{equation}\label{eq:QisN2distQ}
Q=N_2\bigl(\dist^{2}(\ell_0,\,g\cdot\ell_0)-4\bigr)\qquad\text{when }N_2>0 .
\end{equation}

\begin{theorem}\label{thm:domainQ}
The safe zone satisfies $\Dom=\{Q\ge0\}\cap\{U_1^{2}\ge4\}.$
\end{theorem}
\begin{proof}
Fix $g$ and write $\delta=\dist(\ell_0,g\cdot\ell_0)$, so that $g\in\Dom$ means $\delta\ge2$.
Since $a_2$ lies on $g\cdot\ell_0$, we have $\delta\le U_1$ always.

\emph{Case $N_2>0$.} By \eqref{eq:QisN2distQ} and $N_2>0$, the sign of $Q$ is the sign of
$\delta^{2}-4$; hence $Q\ge0\iff\delta\ge2$. If $Q\ge0$ then $U_1\ge\delta\ge2$, so the second
inequality is automatic. Therefore, on $\{N_2>0\}$,
$\Dom=\{Q\ge0\}=\{Q\ge0\}\cap\{U_1^{2}\ge4\}$.

\emph{Case $N_2=0$.} Then $\Gamma(w)=0$ by Lemma~\ref{lem:distQ}, and $N_2=0$, so $Q=0$: the first
inequality holds identically and carries no information. By Lemma~\ref{lem:distQ} again
$\delta=U_1$, so $\delta\ge2\iff U_1^{2}\ge4$. Therefore, on $\{N_2=0\}$,
$\Dom=\{U_1^{2}\ge4\}=\{Q\ge0\}\cap\{U_1^{2}\ge4\}$.
\end{proof}

\subsection{Homogenization}\label{subsec:homogQ}

Theorem~\ref{thm:domainQ} describes $\Dom$ by two polynomial inequalities in the entries of $g$,
and by Lemma~\ref{lem:quat} those entries are quadratic forms in a unit quaternion $q$. Two
polynomial inequalities in $q$ therefore describe $\Dom$. However, they hold only on the unit sphere
$\{|q|=1\}$, which is itself a constraint of the kind we are trying to eliminate.

Homogenization removes it. Every constant appearing in \eqref{eq:GammaQ}, \eqref{eq:UQ} and \eqref{eq:QQ} is
multiplied by the power of $|q|^2$ that makes each summand have the same degree. The resulting
polynomials agree with the originals when $|q|=1$, and are homogeneous, so their signs are
constant along rays through the origin. The condition can then be imposed on all of $\R^4$ at
once, and the unit sphere never appears again.

In the coordinates of Section~\ref{subsec:safezoneQ} the vectors $w$ and $d_2$ are read off from
the entries of $g$. Since $d_1=v_0=(1,0,0)$, the vector $d_2=gv_0$ is the first column of $g$ and
$a_2=2gu_0$ is twice the third, so
\[
d_2=(g_{11},g_{21},g_{31}),\qquad w=(2g_{13},\,2g_{23},\,2g_{33}-2),
\]
and therefore
\[
d_1\cdot d_2=g_{11},\qquad w\cdot d_1=2g_{13},\qquad
n=d_1\times d_2=(0,-g_{31},g_{21}),\qquad N_2=g_{21}^2+g_{31}^2 .
\]
For a unit quaternion $q$ with $g=R(q)$ each $g_{ij}$ is the integer quadratic form $R(q)_{ij}$ of
Section~\ref{subs:quaternionsBasic}, which we abbreviate $R_{ij}$. Set
\[
\mathbb W=\bigl(2R_{13},\;2R_{23},\;2R_{33}-2|q|^2\bigr),
\qquad
\mathbb D_2=\bigl(R_{11},\,R_{21},\,R_{31}\bigr),
\]
which are the homogenizations of $w$ and $d_2$; only the third entry of $w$ carries a constant, and
it becomes $-2|q|^2$. Note $\mathbb W\cdot d_1=2R_{13}$ and
\[
\mathbb W\cdot\mathbb D_2=2R_{13}R_{11}+2R_{23}R_{21}+\bigl(2R_{33}-2|q|^2\bigr)R_{31}.
\]

\begin{definition}\label{def:UhQh}
Define $U_{\mathrm h},Q_{\mathrm h}\in\Z[q_0,q_1,q_2,q_3]$ by
\begin{align*}
U_{\mathrm h}&=|\mathbb W|^2-4R_{13}^2-4|q|^4,\\[2pt]
Q_{\mathrm h}&=\bigl(|\mathbb W|^2-4|q|^4\bigr)\bigl(R_{21}^2+R_{31}^2\bigr)
 -\Bigl[\,|q|^4\bigl(2R_{13}\bigr)^2
   -2\,R_{11}\bigl(2R_{13}\bigr)\bigl(\mathbb W\cdot\mathbb D_2\bigr)
   +\bigl(\mathbb W\cdot\mathbb D_2\bigr)^2\Bigr],
\end{align*}
and let
\[
K=\bigl\{q\in\R^4\setminus\{0\}\ :\ Q_{\mathrm h}(q)\ge0\ \text{ and }\ U_{\mathrm h}(q)\ge0\bigr\}.
\]
\end{definition}

\begin{lemma}\label{lem:homogQ}
\leavevmode
\begin{enumerate}[leftmargin=2.2em,label=(\alph*)]
\item $U_{\mathrm h}$ and $Q_{\mathrm h}$ are homogeneous of degrees $4$ and $8$ respectively, with
integer coefficients.
\item For every unit quaternion $q$, with $g=R(q)$,
\[
U_{\mathrm h}(q)=U_1^2(g)-4\qquad\text{and}\qquad Q_{\mathrm h}(q)=Q(g).
\]
\item For $q\ne0$ and $\hat q=q/|q|$: $q\in K$ if and only if $R(\hat q)\in\Dom$. In particular
$K$ is a cone, invariant under $q\mapsto-q$, and
\[
\Dom=\bigl\{R(q)\ :\ q\in K,\ |q|=1\bigr\}.
\]
\item $U_{\mathrm h}$ and $Q_{\mathrm h}$ are the \emph{unique}
homogeneous polynomials of their degrees satisfying (b).
\end{enumerate}
\end{lemma}
\begin{proof}
(a) Each $R_{ij}$ and $|q|^2$ is a quadratic form with integer coefficients, so every entry of
$\mathbb W$ and of $\mathbb D_2$ is homogeneous of degree $2$, and $\mathbb W\cdot\mathbb D_2$ is
homogeneous of degree $4$. Counting degrees termwise, each summand of $U_{\mathrm h}$ has degree
$4$: $|\mathbb W|^2$, $4R_{13}^2$ and $4|q|^4$ all do. Likewise each summand of $Q_{\mathrm h}$ has
degree $8$: the first product is $4+4$, and inside the bracket the three terms have degrees
$4+4$, $2+2+4$ and $4+4$.

(b) Let $|q|=1$ and $g=R(q)$, which lies in $\SO(3)$ by Lemma~\ref{lem:quat}(a). Then $|q|^2=1$ and
$|q|^4=1$, so $\mathbb W=w$ and $\mathbb D_2=d_2$, and also $d_1\cdot d_2=R_{11}$,
$w\cdot d_1=2R_{13}$, $N_2=R_{21}^2+R_{31}^2$. Substituting into
Definition~\ref{def:UhQh} and comparing with \eqref{eq:GammaQ}, \eqref{eq:UQ}, and \eqref{eq:QQ}, we obtain
\[
U_{\mathrm h}(q)=|w|^2-(w\cdot d_1)^2-4=U_1^2-4,
\qquad
Q_{\mathrm h}(q)=\bigl(|w|^2-4\bigr)N_2-\Gamma(w)=Q .
\]

(c) By (a), $U_{\mathrm h}(q)=|q|^4U_{\mathrm h}(\hat q)$ and
$Q_{\mathrm h}(q)=|q|^8Q_{\mathrm h}(\hat q)$, and $|q|^4,|q|^8>0$, so $q\in K$ if and only if
$\hat q\in K$. By (b) applied to $\hat q$, this holds if and only if $U_1^2(R(\hat q))\ge4$ and
$Q(R(\hat q))\ge0$, which by Theorem~\ref{thm:domainQ} says exactly $R(\hat q)\in\Dom$. Both
degrees are even, so $K$ is invariant under $q\mapsto-q$; and the displayed description of $\Dom$
follows from Lemma~\ref{lem:cover2}(b).

(d)  If $h_1,h_2$ are homogeneous of the same degree $d$ and agree on the
unit sphere, then $h_1(q)=|q|^d h_1(\hat q)=|q|^d h_2(\hat q)=h_2(q)$ for
all $q\ne0$, hence everywhere. So (b) determines $U_{\mathrm h}$ and
$Q_{\mathrm h}$ uniquely.
\end{proof}

\begin{remark}\label{rem:homogTwice}
The same observation is used twice, and it is worth isolating. A homogeneous polynomial of even
degree has constant sign along each ray and is invariant under $q\mapsto-q$; the first property is
what allows a condition on the unit sphere to be imposed on all of $\R^4$, and the second is what
allows a polynomial in $q$ to define a function on $\SO(3)$, where $q$ and $-q$ represent the same
rotation. Hypothesis \ref{H:hom} below asks exactly for this, and $P$, $G$, $U_{\mathrm h}$ and
$Q_{\mathrm h}$ are homogeneous of even degrees $8$, $8$, $4$ and $8$ for the same reason.
\end{remark}

\subsection{Re-formulation of the problem in terms of polynomials}\label{subsec:ReFormulationPolynomials}

\begin{theorem}\label{thm:polyform}
Assume that for each $\ell\in\{1,2,3,4\}$:
\begin{itemize}[leftmargin=1.4em]
\item
      $D_\ell=\diag\bigl(w^\ell_1,\dots,w^\ell_{2\ell+1}\bigr)$ is a diagonal matrix with positive entries $w^\ell_1,\dots,w^\ell_{2\ell+1}$;
\item $T^\ell_{\mathrm h}$ is a $(2\ell+1)\times(2\ell+1)$ matrix whose entries are polynomials
      in $\R[q_0,q_1,q_2,q_3]$;
\item $C_\ell\in\R^{(2\ell+1)\times(2\ell+1)}$ is a symmetric matrix.
\end{itemize}
Define
\[
P(q)\;=\;\sum_{\ell=1}^{4}\bigl(|q|^{2}\bigr)^{4-\ell}
\sum_{i,j}\frac{(C_\ell)_{ij}}{w^\ell_j}\,T^\ell_{\mathrm h}(q)_{ij},
\qquad
G(q)\;=\;-\tfrac{19}{20}\,|q|^{8}-P(q).
\]
Suppose that
\begin{enumerate}[leftmargin=2.6em,label=(H\arabic*)]
\item\label{H:hom} every entry of $T^\ell_{\mathrm h}$ is homogeneous of degree $2\ell$;
\item\label{H:mult} $T^\ell_{\mathrm h}(p*q)=T^\ell_{\mathrm h}(p)\,T^\ell_{\mathrm h}(q)$
      holds in $\R[p_0,\dots,p_3,q_0,\dots,q_3]$;
\item\label{H:orth} $T^\ell_{\mathrm h}(q)^{\top}D_\ell\,T^\ell_{\mathrm h}(q)
      =|q|^{4\ell}D_\ell$ holds in $\R[q_0,\dots,q_3]$;
\item\label{H:psd} $C_\ell\succeq0$;
\item\label{H:trace} $\displaystyle\sum_{\ell=1}^{4}\sum_{i}\frac{(C_\ell)_{ii}}{w^\ell_i}
      \;\le\;\frac{23017}{4096}$;
\item\label{H:cover} $G(q)\ge0$ for every $q$ in the cone
      $K=\{q\in\R^4\setminus\{0\}:\ Q_{\mathrm h}(q)\ge0\ \text{and}\ U_{\mathrm h}(q)\ge0\}$.
\end{enumerate}
Then $N_{\max}\le6$, and consequently at most six pairwise non-overlapping infinite unit cylinders can
touch a unit ball.
\end{theorem}
\begin{proof} 
\emph{Step 1: the matrices $T^\ell_{\mathrm h}$ define maps on $\SO(3)$.}
Let $g\in\SO(3)$. By Lemma~\ref{lem:cover2}(b) there is a unit quaternion $q$ with $R(q)=g$, and by
Lemma~\ref{lem:cover2}(a) the only unit quaternions with this property are $q$ and $-q$. By
\ref{H:hom} every entry of $T^\ell_{\mathrm h}$ is homogeneous of degree $2\ell$, which is even, so
$T^\ell_{\mathrm h}(-q)=T^\ell_{\mathrm h}(q)$. Hence
\[
T^\ell(g)\;:=\;T^\ell_{\mathrm h}(q),\qquad q \text{ any unit quaternion with } R(q)=g,
\]
is a well-defined map $T^\ell:\SO(3)\to\R^{(2\ell+1)\times(2\ell+1)}$.

\emph{Step 2: hypotheses (i) and (ii) of Section~\ref{sec:DBSequence} hold.}
Each $D_\ell$ is diagonal with positive diagonal entries, hence symmetric positive definite, which
is (i). For the first half of (ii), let $g_1,g_2\in\SO(3)$ and choose unit quaternions $p,q$ with
$R(p)=g_1$ and $R(q)=g_2$. By Lemma~\ref{lem:quat}(b) we have $|p*q|^2=|p|^2|q|^2=1$ and
$R(p*q)=R(p)R(q)=g_1g_2$, so $p*q$ is a unit quaternion representing $g_1g_2$. Evaluating
\ref{H:mult} at $(p,q)$,
\[
T^\ell(g_1g_2)=T^\ell_{\mathrm h}(p*q)=T^\ell_{\mathrm h}(p)\,T^\ell_{\mathrm h}(q)
=T^\ell(g_1)\,T^\ell(g_2).
\]
For the second half, let $g\in\SO(3)$ and let $q$ be a unit quaternion with $R(q)=g$. Evaluating
\ref{H:orth} at $q$ and using $|q|=1$,
\[
T^\ell(g)^\top D_\ell\,T^\ell(g)=|q|^{4\ell}D_\ell=D_\ell .
\]

\emph{Step 3: $P$ computes the Delsarte--Bochner sum.}
Let $q$ be a unit quaternion and $g=R(q)$. The prefactors $(|q|^2)^{4-\ell}$ all equal $1$, and
$T^\ell_{\mathrm h}(q)=T^\ell(g)$ by Step 1, so by Remark~\ref{rem:entrywise} applied to the
diagonal matrix $D_\ell$,
\begin{equation}\label{eq:PisTrace}
P(q)=\sum_{\ell=1}^{4}\tr\!\Bigl(D_\ell^{-1}C_\ell\,T^\ell(g)\Bigr).
\end{equation}

\emph{Step 4: the cover hypothesis bounds the supremum.}
Let $g\in\Dom$. By Lemma~\ref{lem:cover2}(b) choose a unit quaternion $q$ with $R(q)=g$; by
Lemma~\ref{lem:homogQ}(c) we have $q\in K$. Hence \ref{H:cover} gives $G(q)\ge0$, that is,
$-\tfrac{19}{20}|q|^{8}-P(q)\ge0$, and $|q|=1$ turns this into $P(q)\le-\tfrac{19}{20}$. By
\eqref{eq:PisTrace},
\[
\sum_{\ell=1}^{4}\tr\!\Bigl(D_\ell^{-1}C_\ell\,T^\ell(g)\Bigr)\;\le\;-\tfrac{19}{20}
\qquad\text{for every }g\in\Dom ,
\]
and taking the supremum over $g\in\Dom$ gives
$\dbsup(\overrightarrow C)\le-\tfrac{19}{20}<0$ for $\overrightarrow C=(C_1,C_2,C_3,C_4)$.

\emph{Step 5: the bound.}
By \ref{H:psd} each $C_\ell$ is symmetric positive semidefinite, so
Theorem~\ref{thm:dbBound} applies with $M=-\tfrac{19}{20}$ and yields
\[
N_{\max}\;\le\;1+\left\lfloor\frac{\sum_{\ell=1}^{4}\tr\bigl(D_\ell^{-1}C_\ell\bigr)}
{19/20}\right\rfloor .
\]
By Remark~\ref{rem:entrywise} the trace sum equals $\sum_{\ell}\sum_i (C_\ell)_{ii}/w^\ell_i$,
which by \ref{H:trace} is at most $\tfrac{23017}{4096}$. Since $x\mapsto\lfloor x\rfloor$ is
non-decreasing,
\[
N_{\max}\;\le\;1+\left\lfloor\frac{23017}{4096}\cdot\frac{20}{19}\right\rfloor
\;=\;1+\left\lfloor\frac{460340}{77824}\right\rfloor\;=\;1+5\;=\;6.
\]

\emph{Step 6: back to cylinders.}
By Section~\ref{sec:reduction}, any $N$ pairwise non-overlapping unit cylinders touching the unit
ball give rise to a non-overlapping subset of $\SO(3)$ with $N$ elements, so
$N\le N_{\max}\le6$.
\end{proof}

 \begin{remark}\label{rem:noMoreSO3}
Theorem~\ref{thm:polyform} is the last statement in the proof that mentions $\SO(3)$, the safe
zone $\Dom$, cylinders, or distances between lines. Everything that follows is devoted to
exhibiting data satisfying \ref{H:hom}--\ref{H:cover}, and each of those six hypotheses is a
statement about polynomials in four real variables (or, in the case of \ref{H:mult}, in eight).
None of them refers to the unit sphere $|q|=1$.

\end{remark}

\section{The certificate}\label{sec:cert}

This section exhibits data satisfying the hypotheses of Theorem~\ref{thm:polyform}. For each $\ell=1,\dots,4$, there are three major components:
\begin{itemize}[leftmargin=1.4em]
\item positive rationals $w^\ell_1,\dots,w^\ell_{2\ell+1}$, forming the diagonal matrix
$D_\ell=\diag(w^\ell_1,\dots,w^\ell_{2\ell+1})$;
\item an integer table defining a $(2\ell+1)\times(2\ell+1)$ matrix $T^\ell_{\mathrm h}(q)$ of
polynomials in $q=(q_0,q_1,q_2,q_3)$, with coefficients in
$\tfrac1{\mathrm{Tden}_\ell}\Z$ for a deposited scale $\mathrm{Tden}_\ell\in\{1,2\}$
($\mathrm{Tden}_\ell=2$ only for $\ell=3$);
\item a symmetric rational matrix $C_\ell\in\Q^{(2\ell+1)\times(2\ell+1)}$, deposited together with
integer witnesses for its positive semidefiniteness.
\end{itemize}

\begin{remark}
No property of this data is assumed, and how it was found plays no role in the proof. It was
obtained by Gram--Schmidt orthogonalization of a basis of harmonic polynomials, as described in
Section~\ref{sec:motivation}; neither the reader nor the verifier needs this. The checker sees only
the deposited tables, and must therefore establish the properties below directly.
\end{remark}

The standing assumption of Theorem~\ref{thm:polyform} that each $w^\ell_i$ is positive is read off
from the deposited numerators and denominators. Hypotheses \ref{H:hom}, \ref{H:mult} and
\ref{H:orth} are finite computations on the tables, \ref{H:psd} follows from the deposited
witnesses by Proposition~\ref{prop:psd} below, \ref{H:trace} is a finite computation, and
\ref{H:cover} is the subject of Section~\ref{sec:cover}.

\subsection{The identities \ref{H:hom}--\ref{H:orth}}\label{subsec:maps}

Hypothesis \ref{H:hom} is an inspection of the exponents in the deposited table: every row of
$T^\ell_{\mathrm h}$ carries exponents summing to $2\ell$.

Both \ref{H:mult} and \ref{H:orth} are identities between polynomials with integer coefficients
once the deposited scales are cleared. Writing
$A^\ell=\mathrm{Tden}_\ell\,T^\ell_{\mathrm h}$, which has integer entries, and
$\widehat D_\ell=m_\ell D_\ell$, where $m_\ell$ is a common denominator of
$w^\ell_1,\dots,w^\ell_{2\ell+1}$, they read
\[
\mathrm{Tden}_\ell\,A^\ell(p*q)=A^\ell(p)\,A^\ell(q),
\qquad
A^\ell(q)^\top \widehat D_\ell\,A^\ell(q)=\mathrm{Tden}_\ell^{2}\,|q|^{4\ell}\,\widehat D_\ell ,
\]
so each amounts to comparing two polynomial expansions with integer coefficients, coefficient by
coefficient. No approximation and no choice of inner product enters.

\subsection{Positive semidefiniteness \ref{H:psd}}\label{subsec:C}

The matrices $C_1,\dots,C_4$ are given here as rational data; no motivation for their values is
needed, and their construction is sketched in Section~\ref{sec:motivation}. Their positive
semidefiniteness is not assumed. It follows from the deposited integer witnesses by the following
two elementary facts.

\begin{lemma}\label{lem:domPSD}
Let $E\in\R^{n\times n}$ be symmetric with
$E_{ii}\ge\sum_{j\ne i}|E_{ij}|$ for every $i$. Then $E\succeq0$.
\end{lemma}
\begin{proof}
Let $v\in\R^{n}$. Using $|v_iv_j|\le\tfrac12(v_i^{2}+v_j^{2})$ and then the symmetry of $E$,
\[
\Bigl|\sum_{i\ne j}E_{ij}v_iv_j\Bigr|
\;\le\;\sum_{i\ne j}|E_{ij}|\,\tfrac12\bigl(v_i^{2}+v_j^{2}\bigr)
\;=\;\sum_{i}v_i^{2}\sum_{j\ne i}|E_{ij}|
\;\le\;\sum_{i}E_{ii}v_i^{2},
\]
so $v^{\top}Ev=\sum_i E_{ii}v_i^{2}+\sum_{i\ne j}E_{ij}v_iv_j\ge0$.
\end{proof}

\begin{proposition}\label{prop:psd}
Write $\widehat C_\ell:=2^{20}C_\ell$ for the deposited integer matrix, and let $S_\ell$ and
$W_\ell$ be the deposited positive integer scalar and integer matrix. If
\[
E_\ell\;=\;S_\ell\,\widehat C_\ell\;-\;W_\ell^{\top}W_\ell
\]
is symmetric and diagonally dominant, in the sense that
$(E_\ell)_{ii}\ge\sum_{j\ne i}|(E_\ell)_{ij}|$ for every $i$, then $C_\ell$ is symmetric positive
semidefinite.
\end{proposition}
\begin{proof}
By Lemma~\ref{lem:domPSD}, $E_\ell\succeq0$. Also $W_\ell^{\top}W_\ell\succeq0$, since
$v^{\top}W_\ell^{\top}W_\ell v=|W_\ell v|^{2}\ge0$ for every $v$. Hence
\[
2^{20}S_\ell\,C_\ell\;=\;S_\ell\,\widehat C_\ell\;=\;E_\ell+W_\ell^{\top}W_\ell\;\succeq\;0,
\]
and dividing by the positive scalar $2^{20}S_\ell$ gives $C_\ell\succeq0$. Symmetry of $C_\ell$ is
symmetry of the deposited $\widehat C_\ell$.
\end{proof}

For the deposited data $S_\ell=2^{26}$, and every entry of $S_\ell\widehat C_\ell$ and of
$W_\ell^{\top}W_\ell$ is an integer of absolute value below $2^{46}$, so the hypothesis of
Proposition~\ref{prop:psd} is a check in exact $64$-bit integer arithmetic. Note that diagonal
dominance forces $(E_\ell)_{ii}\ge0$, the right-hand side being a sum of absolute values.

\subsection{The trace sum \ref{H:trace}}\label{subsec:trace}

For the deposited data the trace sum is not merely bounded by $\tfrac{23017}{4096}$ but equal to it:
\begin{equation}\label{eq:trsumcheck}
\sum_{\ell=1}^{4}\tr\bigl(D_\ell^{-1}C_\ell\bigr)
=\sum_{\ell=1}^{4}\sum_{i=1}^{2\ell+1}\frac{(C_\ell)_{ii}}{w^\ell_i}=\frac{23017}{4096},
\end{equation}
which is a finite computation with rational numbers. In particular \ref{H:trace} holds.

\subsection{The certificate polynomial}\label{subsec:P}

The polynomial $P$ of Theorem~\ref{thm:polyform}, formed from the deposited $C_\ell$, the tables
$T^\ell_{\mathrm h}$ and the weights, has rational coefficients and is homogeneous of degree $8$:
by \ref{H:hom} each $T^\ell_{\mathrm h}(q)_{ij}$ is homogeneous of degree $2\ell$, and
$2(4-\ell)+2\ell=8$ for every $\ell$. The same holds for $G=-\tfrac{19}{20}|q|^8-P$.

\begin{remark}\label{rem:deposited}
Clearing denominators makes $G$ integral. Put $G_{\mathrm{den}}=40\cdot2^{20}=41943040$; then all
$35$ coefficients of $G_{\mathrm{den}}\,G$ are integers of absolute value $<2^{32}$. A fingerprint
identifies the polynomial: since $R(1,0,0,0)$ is the identity, $P(1,0,0,0)$ is the trace sum
\eqref{eq:trsumcheck}, so the coefficient of $q_0^{8}$ in $G_{\mathrm{den}}G$ equals
\[
-\Bigl(\tfrac{19}{20}+\tfrac{23017}{4096}\Bigr)G_{\mathrm{den}}=-275539968 .
\]
Any implementation that reconstructs $G$ from the deposited data must reproduce these $35$
integers.
\end{remark}

\section{A computer-assisted proof of hypothesis (H6)}\label{sec:cover}

By Theorem~\ref{thm:polyform} and Section~\ref{sec:cert}, the proof of Theorem~\ref{thm:main} is
complete once \ref{H:cover} is established for the deposited certificate, that is, once we show
\begin{equation}\label{eq:gcone}
G(q)\ \ge\ 0\qquad\text{on}\qquad
K=\{q\ne0:\ Q_{\mathrm h}(q)\ge0\ \text{and}\ U_{\mathrm h}(q)\ge0\}.
\end{equation}
Because $G$, $Q_{\mathrm h}$ and $U_{\mathrm h}$ are homogeneous and even in
$q$, it suffices to certify \eqref{eq:gcone} on the four cube-faces
$\{q_i=1,\ |q_j|\le1\ (j\ne i)\}$, $i=0,\dots,3$, whose rays cover all of $\R^4\setminus\{0\}$ up
to $q\mapsto-q$.

\subsection{Algorithm}
For each $i\in\{0,1,2,3\}$, we make a separate problem. Set $q_i=1$, and re-write
$G$, $Q_{\mathrm h}$, and $U_{\mathrm h}$ as polynomials in three variables. In each problem, we can re-label these variables and call them
$x$, $y$, and $z$. The theorem will be proved if
we manage to partition $[-1,1]^3$ into finitely many boxes of the form $[L_x,R_x]\times[L_y,R_y]\times [L_z,R_z]$, where none of the open intervals
$(L_x,R_x)$, $(L_y,R_y)$, $(L_z,R_z)$ contains $0$, and to prove that on each box at least
one of the following three inequalities is satisfied:
\begin{align}
U_{\mathrm h}(x,y,z)&< 0,\quad Q_{\mathrm h}(x,y,z)<0,\quad\text{or}\quad G(x,y,z)\geq 0.
\label{eqn:threeInequalities}
\end{align}
This is a technique for proving that a polynomial $F(x,y,z)$ has a constant sign on $[L_x,R_x]\times[L_y,R_y]\times [L_z,R_z]$.
 First, split $F$ into a sum of monomials \begin{align*}F(x,y,z)&=\sum_j \gamma_j x^{a_j}y^{b_j}z^{c_j}.
\end{align*}
Then, for each $j$, each of the maximum $M_j$ and the minimum $m_j$ of the monomial $\gamma_jx^{a_j}y^{b_j}z^{c_j}$ is attained at a corner of the box. This is due to the fact that each factor is monotone on its interval, because $0$ is not interior to it. Hence, $m_j$ and $M_j$ can be computed by hand, or by a computer. Define $M_P=\sum_jM_j$ and $m_P=\sum_jm_j$. If it happens that $M_P<0$, then we can guarantee that $F(x,y,z)<0$ on the box $[L_x,R_x]\times[L_y,R_y]\times [L_z,R_z]$. Similarly, if it happens that
$m_P\geq 0$, then we are allowed to conclude that $P(x,y,z)\geq 0$.

None of this requires rounding. After the scaling of Remark~\ref{rem:deposited} the coefficients
$\gamma_j$ are integers, and every corner produced by repeated halving has the form $k/2^{d}$ with
$k,d$ integers, so each $M_j$ and $m_j$, and hence $M_P$ and $m_P$, is a rational number whose
denominator is a power of two. Both tests ask only for a sign, and a sign is unchanged by
multiplication by a positive number: multiplying through by $2^{8d}$, the total degree being $8$,
turns each of them into a comparison between two integers. At the depths reached here those
integers stay below $2^{120}$.

Our program takes exactly that route. Every quantity it forms is a dyadic rational,
stored as a $128$-bit integer numerator together with the exponent of the power of two in the
denominator. The denominator itself is never formed, so its size places no restriction on the depth
of the subdivision. Sums, products and comparisons of such numbers are carried out exactly, and
every arithmetic operation is guarded by an overflow test. The run reported below completes with no
overflow signaled. Consequently every sign decided during the run is the exact sign of the
quantity in question, and no floating-point number occurs anywhere in the computation.

One refinement keeps the number of boxes manageable: rather than a uniform grid, we subdivide adaptively.
We first split $[-1,1]^3$ into $8$ boxes by midpoints, after which $0$ is always an endpoint
and never interior. Then for each box we try to see if one of \eqref{eqn:threeInequalities} is satisfied. If it is, we say that the box is certified. If it is not, then we split the box into 8 equally sized smaller boxes. We keep going until the number of uncertified boxes becomes $0$. Nothing guarantees in advance that this happens: if some box could never be certified, the splitting would continue forever and we would have proved nothing. That it does stop is the whole content of the computation. In our case, the process was finite.

The above algorithm with target $-\tfrac{19}{20}$ terminates. About
$3.4$ million boxes are tested in all: $1\,766\,376$ of them were settled due to $G\geq 0$, $1\,188\,608$ had $U_{\mathrm h}<0$  or $Q_{\mathrm h}<0$, and the rest were split further. No box remains, and none fails. The boxes that are
never split are the $2\,954\,984$ cases referred to in the abstract; they are what the four
partitions consist of, and each is settled by one of the three inequalities
\eqref{eqn:threeInequalities}. Of the $1\,188\,608$ boxes for which $U_{\mathrm h}<0$  or $Q_{\mathrm h}<0$, $96\,364$ had 
$U_{\mathrm h}<0$ and $1\,092\,244$ had $Q_{\mathrm h}<0$. Therefore, almost all of the decisions were made
by $Q_{\mathrm h}$. The inequality $U_{\mathrm h}<0$ earns its place near the parallel locus, where 
$Q_{\mathrm h}$ vanishes identically and can decide nothing.

The subdivision does not go deep: the smallest boxes produced have half-width $2^{-10}$.  

Hypotheses 
\ref{H:hom}--\ref{H:cover} are established for the deposited certificate, so Theorem~\ref{thm:polyform} gives 
$N_{\max}\leq6$, and with it Theorem~\ref{thm:main}.

\section{Where the certificate comes from}\label{sec:motivation}

Nothing in this section is used anywhere in the proof. It is included because the data of
Section \ref{sec:cert} would otherwise arrive with no explanation at all, and because a reader who wants
to construct a certificate of their own will need it. A reader who only wants to check the theorem
may stop at Section \ref{sec:cover}.

\subsection{The bound}
By Section \ref{sec:reduction}, a configuration of non-overlapping touching cylinders is an independent
set in the graph on $\SO(3)$ in which $g$ and $g'$ are adjacent when $g^{-1}g'\notin\Dom$. Bounding
independence numbers of graphs invariant under a group action is the subject of Delsarte's linear
programming method and of its semidefinite refinement; on a compact group the objects that produce
such bounds are the functions of positive type. Theorem~\ref{thm:dbBound} is one such bound,
stated in the narrow form we need so that it can be proved in half a page. The general framework,
and its use for packing problems on spheres, is that of Bachoc and
Vallentin~\cite{BV}.

\subsection{Why sums of traces of representations}
The irreducible real representations of $\SO(3)$ are its actions on the spaces $H_\ell$ of
harmonic polynomials on $\R^3$, homogeneous of degree $\ell$, of dimension $2\ell+1$. Bochner's
theorem for compact groups identifies the continuous functions of positive type on $\SO(3)$ as the
sums
\[
g\;\longmapsto\;\sum_{\ell\ge0}\tr\bigl(A_\ell\,\pi^\ell(g)\bigr),
\qquad A_\ell\succeq0,
\]
over those representations $\pi^\ell$. Truncating the sum at $\ell\le L$ turns the search for a
good bound into a finite semidefinite program in the matrices $A_\ell$. The maps
$T^\ell_{\mathrm h}$ of Section \ref{subsec:maps} are these representations written in an explicit
basis, and the matrices $C_\ell$ are the corresponding $A_\ell$. Theorem~\ref{thm:dbBound} itself
uses none of this: it assumes only the two identities (i)--(ii) of Section \ref{sec:DBSequence}, which is
why those, and not irreducibility or continuity, are what Section \ref{sec:cert} verifies.

\subsection{The basis, and why the data is rational}
Fix $\ell$ and take any basis of $H_\ell$ with integer coefficients. Apply Gram--Schmidt with
respect to
\[
\langle p,r\rangle\;=\;\frac1{4\pi}\int_{\Sph^2}p\,r\,d\sigma .
\]
Every such inner product is rational, due to 
\begin{align*}
\frac1{4\pi}\int_{\Sph^2}x^ay^bz^c\,d\sigma&=
\left\{\begin{array}{ll}
\frac{(a-1)!!\,(b-1)!!\,(c-1)!!}{(a+b+c+1)!!}, & a,b,c \text{ all even},\\
0,& \text{otherwise},
\end{array}\right. 
\end{align*}
which is easy to prove (this is a standard integral and can be calculated using spherical coordinates).
Hence, the process never leaves
$\Q$. Clearing denominators gives integer polynomials $h^\ell_1,\dots,h^\ell_{2\ell+1}$. The Gram
matrix is then diagonal by construction, and it is the matrix $D_\ell$ of Section \ref{sec:cert}:
$\langle h^\ell_i,h^\ell_j\rangle=\delta_{ij}w^\ell_i$ with $w^\ell_i>0$.

Since $\SO(3)$ acts on $H_\ell$, for each $g=R(q)$ there are unique coefficients with
\begin{align}
h^\ell_j\bigl(R(q)^\top x\bigr)&=\sum_{i=1}^{2\ell+1}T^\ell_{\mathrm h}(q)_{ij}\,h^\ell_i(x),
\label{eq:defid}
\end{align}
and these are the deposited tables. From \eqref{eq:defid} the three properties
\ref{H:hom}--\ref{H:orth} are immediate: homogeneity of degree $2\ell$ is a degree count, since the
entries of $R(q)$ are quadratic in $q$; \ref{H:mult} is the composition law of the action; and
\ref{H:orth} is the invariance of $\sigma$ under $\SO(3)$, which makes each $g$ act by an isometry
of $\langle\cdot,\cdot\rangle$. We stress that the proof does not proceed this way. The harmonic
polynomials $h^\ell_i$ are not deposited and the program never sees them. Section \ref{sec:cert}
establishes \ref{H:hom}--\ref{H:orth} directly from the tables, so no property of the construction
above has to be trusted.

\subsection{The matrices $C_\ell$}
The $C_\ell$ play the role of the matrices $A_\ell$ above: they are the free variables of the
semidefinite program, and the bound that Theorem~\ref{thm:dbBound} returns depends on them through
the ratio of $\dbsup(\overrightarrow C)$ to the trace sum $\sum_\ell\tr(D_\ell^{-1}C_\ell)$. The
deposited values are dyadic, with denominators dividing $2^{20}$.

A numerical solution of a semidefinite program is not usable as it stands, and rounding a feasible
point need not leave it feasible. The paper therefore takes nothing at all from the computation
that produced these matrices: positive semidefiniteness is re-established on the rounded values by
the integer witnesses of Proposition~\ref{prop:psd}, the trace sum by \eqref{eq:trsumcheck}, and the bound on $\Dom$ by
the computation of Section \ref{sec:cover}.

\subsection{The two constants}
The truncation $L=4$ and the threshold $M=-\tfrac{19}{20}$ are the only arbitrary-looking choices
in Section \ref{sec:Strategy}, and neither is delicate.

Larger $L$ gives a richer family of positive-type functions and hence a bound at least as good;
$L=4$ is what we found sufficient.

For $M$ there is a whole interval of admissible values. Any $M$ with
$\dbsup(\overrightarrow C)\le M<0$ may be used, and by \eqref{eq:dbBound} the resulting bound is
$6$ as soon as $\tfrac{23017}{4096}\big/(-M)<6$, that is, as soon as
$-M>\tfrac{23017}{24576}=0.93656\ldots$. On the other side, $M$ must be at least
$\dbsup(\overrightarrow C)$. The rational point $q^{*}=\bigl(1,\tfrac{2365}{4096},0,0\bigr)$ lies
in the cone $K$ of Definition~\ref{def:UhQh}. The following are satisfied:
$U_{\mathrm h}(q^{*})=55008914419/2^{46}>0$ and $Q_{\mathrm h}(q^{*})=0$. The exact 
rational evaluation gives
$P(q^{*})/|q^{*}|^{8}=-0.98406\ldots$, so
$\dbsup(\overrightarrow C)\ge-0.98406\ldots$ and no $M$ below that value can be used.
Numerical calculations show $\dbsup(\overrightarrow C)\approx-0.9838$. So any
\[
M\in\bigl[\,\dbsup(\overrightarrow C),\;-0.93656\ldots\,\bigr)
\qquad\text{with }\dbsup(\overrightarrow C)\approx-0.9838,
\]
yields the theorem, and $-\tfrac{19}{20}=-0.95$ is a convenient rational lying inside it with room
on both sides. The cover of Section \ref{sec:cover} certifies the left-hand requirement;
the arithmetic in the proof of Theorem~\ref{thm:polyform} certifies the right-hand one.

\appendix

\section{Deposited data and standalone verifier}\label{app:data}

The certificate is in a single plain-text file in a self-describing format, designed to
be read by a short standalone parser. Each record consists of a key, written
\texttt{[key]Name[/key]}, followed by a value block \texttt{[value]}\,$\dots$\,\texttt{[/value]}.
A value block holds one or more rows. Each row, and each entry within a row, is delimited by
\texttt{\{} and \texttt{\}}. Every value is therefore a nested list of integers of depth
two, including those that carry a single number, which appear as one row with one entry. Every key
ends in a four-digit index $000\ell$.

The certificate consists of $28$ records, seven for each $\ell=1,\dots,4$:
\begin{itemize}[leftmargin=1.4em]
\item \texttt{WNum000$\ell$} and \texttt{WDen000$\ell$}: the weights $w^\ell_i$ as
numerator/denominator pairs, forming \[D_\ell=\diag(w^\ell_1,\dots,w^\ell_{2\ell+1});\]
\item \texttt{Tpoly000$\ell$} together with the scale \texttt{Tden000$\ell$}. Each row has 7 numbers
$i$, $j$, $e_0$, $e_1$, $e_2$, $e_3$, $c$. These seven numbers mean that the polynomial $T^\ell_{\mathrm h}(q)_{ij}$
contains the term  
$(c/\mathrm{Tden}_\ell)\,q_0^{e_0}q_1^{e_1}q_2^{e_2}q_3^{e_3}$. The polynomial $T^\ell_{\mathrm h}(q)_{ij}$ is the sum 
over the rows carrying that $i,j$;
\item \texttt{Cscaled000$\ell$} $=2^{20}C_\ell$, as the rows of an integer matrix;
\item \texttt{Swit000$\ell$} and \texttt{Wwit000$\ell$}: the integer scalar $S_\ell$ and the
integer matrix $W_\ell$ of Proposition~\ref{prop:psd}.
\end{itemize}

Nothing further is deposited. Every other quantity in the proof is derived from these $28$
records: the reciprocals $1/w^\ell_i$ from \texttt{WNum}/\texttt{WDen}; the certificate polynomial
$P$ of Theorem~\ref{thm:polyform}, and hence the cover polynomial $G=-\frac{19}{20}|q|^8-P$, from the
$C_\ell$, the $T^\ell_{\mathrm h}$ tables and the weights; and $U_{\mathrm h}$, $Q_{\mathrm h}$
from the matrix $R(q)$ of Section \ref{sec:quat} by the formulas of Definition~\ref{def:UhQh}. The two
rational constants of the argument, the threshold $-\tfrac{19}{20}$ and the trace sum
$\tfrac{23017}{4096}$, are those of Section \ref{sec:Strategy} and Theorem~\ref{thm:polyform}.

On these records the standalone verifier confirms:
\begin{itemize}[leftmargin=1.4em]
\item \ref{H:hom} and the positivity of the weights, by inspection of the exponents in
\texttt{Tpoly000$\ell$} and of the signs in \texttt{WNum000$\ell$}, \texttt{WDen000$\ell$};
\item \ref{H:mult} and \ref{H:orth}, by expanding both sides of the two identities of
Section~\ref{subsec:maps} and comparing them coefficient by coefficient;
\item the hypothesis of Proposition~\ref{prop:psd}, forming
$E_\ell=S_\ell\bigl(2^{20}C_\ell\bigr)-W_\ell^{\top}W_\ell$ and testing symmetry and diagonal
dominance; every intermediate product is an integer of absolute value below $2^{46}$, so this runs
in $64$-bit integer arithmetic;
\item \ref{H:trace}, by evaluating $\sum_{\ell}\sum_i (C_\ell)_{ii}/w^\ell_i$ and comparing it with
$\tfrac{23017}{4096}$;
\item \ref{H:cover}, by the subdivision of Section~\ref{sec:cover}, reproducing the counts reported
there: every box certified, and no failures.
\end{itemize}

Every one of these checks is exact. The verifier uses no floating-point arithmetic. Integers are
held in $64$- or $128$-bit machine words. For rational numbers, we use two representations. If the number is not dyadic, then we 
store the numerator and denominator. In the dyadic case the denominator is a power of two, and we keep only its exponent. Addition, multiplication and comparison are implemented so that no rounding can occur, and every
operation is guarded by an overflow test. A run in which any such test fires would report the fact and
discard the output. All of the tests were successful.


\begin{thebibliography}{99}
\bibitem{BV} C.~Bachoc and F.~Vallentin, \emph{New upper bounds for kissing numbers from semidefinite
programming}, J.\ Amer.\ Math.\ Soc.\ \textbf{21} (2008), 909--924.
\bibitem{BW} P.~Bra\ss{} and C.~Wenk, \emph{On the number of cylinders touching a ball}, Geom.\
Dedicata \textbf{81} (2000), 281--284.
\bibitem{F} M.~Firsching, \emph{Optimization methods in discrete geometry}, PhD thesis, FU Berlin
(2016).
\bibitem{HS} A.~Heppes and L.~Szab\'o, \emph{On the number of cylinders touching a ball}, Geom.\
Dedicata \textbf{40} (1991), 111--116.
\bibitem{K} W.~Kuperberg, \emph{How many unit cylinders can touch a unit ball?}, Problem 3.3, DIMACS
Workshop on Polytopes and Convex Sets, Rutgers University (1990).
\bibitem{MRCode2026} I.~Mati\'c, R.~Radoi\v ci\'c, \emph{Source code and printout of cases for
a computer-assisted proof of Kuperberg's six-cylinder conjecture},
\verb+https://github.com/maticivan/cyl+ (2026).
\bibitem{OS} O.~Ogievetsky and S.~Shlosman, \emph{The six cylinders problem: $\mathbb{D}_3$-symmetry
approach}, arXiv:1805.09833; \emph{Extremal cylinder configurations}, arXiv:1812.09543,
arXiv:1902.08995.
\bibitem{SW} K.~Sch\"utte and B.~L.~van der Waerden, \emph{Das Problem der dreizehn Kugeln}, Math.\
Ann.\ \textbf{125} (1953), 325--334.
\bibitem{Yar} O.~Yardimci, \emph{On the number of cylinders touching a sphere}, PhD dissertation,
Auburn University (2019); the seven-cylinder radius bound is joint with A.~Bezdek.
\end{thebibliography}
\end{document}